\documentclass[a4paper,11pt]{amsart}
\usepackage[latin1]{inputenc}

\usepackage{amsfonts, amssymb, amsmath, amscd, latexsym, graphicx, psfrag, color, float, xspace, enumerate}
\usepackage[all]{xy}

\usepackage{todonotes}
\usepackage{hyperref}
\usepackage{mathrsfs}

\newenvironment{enumerate1}
{\begin{enumerate}[\upshape (1)]}
{\end{enumerate}}

\newtheorem{dummy}{dummy}[section]

\newtheorem{corollary}[dummy]{Corollary}
\newtheorem{proposition}[dummy]{Proposition}
\theoremstyle{definition}
\newtheorem{definition}[dummy]{Definition}

\newtheorem{examples}[dummy]{Examples}
\newtheorem{remark}[dummy]{Remark}

\newtheorem{assumption}[dummy]{Assumption}

\newcommand{\bA}{\mathbb{A}}
\newcommand{\bC}{\mathbb{C}}

\newcommand{\bG}{\mathbb{G}}

\newcommand{\bN}{\mathbb{N}}

\newcommand{\bQ}{\mathbb{Q}}
\newcommand{\bR}{\mathbb{R}}
\newcommand{\bZ}{\mathbb{Z}}

\newcommand{\cA}{\mathcal{A}}

\newcommand{\cL}{\mathcal{L}}

\newcommand{\cO}{\mathcal{O}}

\newcommand{\cS}{\mathcal{S}}

\newcommand{\cX}{\mathcal{X}}
\newcommand{\cY}{\mathcal{Y}}

\newcommand{\fX}{\mathfrak{X}}
\newcommand{\fY}{\mathfrak{Y}}

\newcommand{\Spec}{\mathrm{Spec}\,}

\newcommand{\Hom}{\mathrm{Hom}}

\newcommand{\Div}{\mathrm{Div}}

\renewcommand{\log}{\mathrm{log}}

\newcommand{\gp}{\mathrm{gp}}
\newcommand{\an}{\mathrm{an}}

\newcommand{\et}{\mathrm{\acute{e}t}}
\renewcommand{\top}{\mathrm{top}}

\newcommand{\esp}{\mathbf{Esp}}

\newcommand{\St}{\mathrm{Stack}}

\newcommand{\sym}{\mathrm{SymMon}}

\newcommand{\arr}{\to}

\newcommand{\pt}{\mathbf{pt}}

\newcommand{\sm}{_{\mathrm{sm}}}

\newcommand{\gm}{\bG_{\mathrm m}}

\newcommand{\double}{\rightrightarrows}

\newcommand\radice[2][\relax]{\hspace{-1.5pt}\sqrt[\uproot{2}#1]{#2}}

\renewcommand\projlim{\varprojlim}

\begin{document}

\title[A general formalism for logarithmic structures]{A general formalism\\for logarithmic structures}

\author{Mattia Talpo}
\address{Department of Mathematics\\
Simon Fraser University\\
8888 University Drive\\
Burnaby BC\\
V5A 1S6 Canada, and Pacific Institute for the Mathematical Sciences\\ 4176-2207 Main Mall \\ Vancouver BC\\ V6T 1Z4 Canada}

\email{mtalpo@sfu.ca}

\author[Angelo Vistoli]{Angelo Vistoli$^\dagger$}
\address{Scuola Normale Superiore\\Piazza dei Cavalieri 7\\
56126 Pisa\\ Italy}

\email{angelo.vistoli@sns.it}

\thanks{$^\dagger$Partially supported by research funds from the Scuola Normale Superiore}

\maketitle

\begin{abstract}
We extend the formalism of ``log spaces'' of \cite{molcho} to topoi equipped with a sheaf of monoids, and discuss Deligne--Faltings structures and root stacks in this context.
\end{abstract}

\setcounter{tocdepth}{1}

\tableofcontents

\section{Introduction}

\subsubsection*{Logarithmic structures} Logarithmic structures were introduced in the late '80s, in the work of Fontaine, Illusie, Deligne and Faltings, and studied systematically starting with the work of K. Kato \cite{kato}. The motivating philosophy was that, sometimes, a degenerate object behaves like a smooth one, when equipped with the correct log structure. Since then, these objects had a number of applications, especially (but not only) in the construction of nice compactifications of moduli spaces, where adding a log structure often produces compact moduli spaces, and/or isolates the ``main component'', i.e. the closure (or its normalization) of the locus of smooth objects (a few examples are \cite{katof, olsson3, olsson1, olsson2} - see also \cite[Sections 4 and 10]{abramovich}). Recently, log geometry was also employed to define Gromov-Witten invariants for singular targets \cite{gross-siebert-GW, abramovichgw, chen} and to formulate a version of mirror symmetry, in the form of the ``Gross--Siebert program'' \cite{gross-siebert, gross-siebert-II, gross-siebert-III}.

\subsubsection*{The language of symmetric monoidal functors} Traditionally (in Kato's language), a log structure on a scheme $X$ consists of a sheaf of monoids $M$ on the small \'etale site of $X$, together with a homomorphism $\alpha\colon M\to \cO_X$ that induces an isomorphism $\alpha^{-1}\cO_X^\times\to \cO_X^\times$ between units. An alternative point of view was introduced in \cite{borne-vistoli}, where a log structure is seen as a symmetric monoidal functor $L\colon A\to \Div_X$, called a ``Deligne--Faltings structure'', where $A$ is a sheaf of sharp monoids on the small \'etale site of $X$, and $\Div_X$ is the symmetric monoidal fibered category of line bundles with a global section $(L,s)$. This second language is sometimes more convenient; for example, it allows to define parabolic sheaves (this is the purpose for which it was created).

Also in \cite{borne-vistoli}, inspired by ideas of M. Olsson, the authors introduce a sequence of stacks $\radice[n]X$ for a fine saturated logarithmic scheme $X$; then in \cite{TV} we consider the ``infinite root stack'' $\radice[\infty]X := \projlim \radice[n]X$, and show how this reflects the geometry of $X$ very closely.

\subsubsection*{Generalizations to other geometric objects} It is readily apparent that Kato's definition makes sense in more general contexts, where the scheme $X$ is replaced for example by a ringed topos, and this was certainly remarked several times in the existing literature. In the recent preprint \cite{molcho} the authors introduce and systematically study log structures in Kato's language, in an abstract setting of a ``category of spaces''. In their language, this is a category of topological spaces with some extra structure and a monoid object $\bA^1$, satisfying a list of axioms. The sheaf $\cO_X$ is recovered as the sheaf of morphisms towards $\bA^1$, and the definition of a log structure is easily translated in this abstract setting. The authors apply this point of view to define ``manifolds with corners'' as certain log smooth ``positive differentiable spaces'', where the monoid object is the space of nonnegative real numbers, equipped with multiplication.

Subsequently, in \cite{TVnew}, whose aim is to show that the Kato--Nakayama space is a transcendental analogue of the infinite root stack, we also construct infinite root stacks for fine saturated log analytic spaces and log topological spaces.

\subsubsection*{The present paper} The main purpose of this note is to frame all of the constructions above in a general context. 

To start with, we study log structures on sites with a sheaf of monoids. We reprove several of the results of \cite{borne-vistoli} in this rarefied setup: for example, we prove the equivalence of Kato's definition with that of ``Deligne--Faltings structures'' of \cite{borne-vistoli}.  Furthermore, we show how the theory of charts also works, thus allowing us to define some of the basic properties of log structures (such as being fine, saturated, free, \dots).

Then we generalize the formalism of \cite{molcho} to sites, instead of topological spaces. This allows us to include, for example, the case of schemes, which was outside of the setup of \cite{molcho}. In this context we get a good theory of finite and infinite root stacks; locally, these can be constructed explicitly as quotients of groupoids in the site itself, or, in the case of infinite root stacks, as projective limits of such quotients.

The resulting notions for analytic and topological spaces are the same used in \cite{TVnew}.

Hopefully this should be useful for people looking to generalize log structures to other contexts: for example, it should give a good theory of log non-archimedean analytic spaces, and provide analogues of the Kato--Nakayama space in this context.

\subsubsection*{Outline}

In Section \ref{sec:topoi} we define log structures on a topos equipped with a monoid object (what we call a ``monoided topos''), both in the sense of Kato and in the ``Deligne--Faltings language'', and we point out the connection between the two notions.
Section \ref{subsubsec:coherent} is about local models for log structures: we define Kato charts and Deligne--Faltings charts, and identify the spaces and stacks that parametrize them (working in a ``category of spaces'' that satisfies some axioms, as in \cite{molcho}).
We discuss root stacks in Section \ref{sec:root.stacks}, and finally, in Section \ref{sec:2}, we describe some examples of the formalism developed in the paper.

\subsection*{Acknowledgements}

We are grateful to the referee for useful comments.

\section{Log structures}\label{sec:topoi}

We start by defining log structure on a topos equipped with a commutative monoid object. By ``topos'', we will always mean a Grothendieck topos, i.e. the category of sheaves of sets on some Grothendieck site. We will interchangeably refer to an ``objects in the topos $\fX$'' or to a ``sheaf on $\fX$''.

If $\fX$ is a topos, we denote by $\pt$ its terminal object. A \emph{covering of $\fX$} is a set $\{U_{i}\}$ of objects of $\fX$ such the morphism $\bigsqcup_{i}U_{i} \arr \pt$ is an epimorphism.
If $U$ is an object of $\fX$, we denote by $(\fX/U)$ the comma category, which is well known to be a topos. If $A$ is a monoid in $\fX$, we will denote by $A(U)$ the monoid of arrows $U \arr A$.

\begin{definition}
A \emph{monoided topos} is a pair $(\fX,\cO)$ consisting of a topos $\fX$ and a monoid $\cO$ in $\fX$.
\end{definition}

We denote by $\cO^\times$ the subobject of $\cO$ consisting of invertible sections.

\begin{remark}
The strange term ``monoided'' is a back-formation inspired from ``ringed''. We will reserve ``monoidal'' for symmetric monoidal categories.
\end{remark}

\begin{examples}
We will later give some examples for this formalism in concrete situations, that the reader might want to have in mind right away. In all of them, the operation on the sheaf is given by multiplication.

If $X$ is a scheme (or algebraic space) over a ring $k$, we can consider the topos of sheaves on the small \'etale site of $X$, with its structure sheaf.

If $X$ is a complex analytic space, we can consider the topos of sheaves on $X$ for the analytic topology, and its sheaf of analytic functions.

Finally, if $X$ is a topological space, we can consider the topos of sheaves on $X$ and the sheaf of continuous complex-valued functions.

All these sheaves of commutative rings are seen as sheaves of monoids via multiplication.
\end{examples}

\begin{definition}
A \emph{Kato pre-log structure} on a monoided topos $(\fX,\cO)$ is a pair $(M,\alpha)$ consisting of a monoid $M$ in $\fX$ and a homomorphism $\alpha\colon M\to \cO$.

A Kato pre-log structure is called a \emph{Kato log structure} if the restriction $\alpha|_{\alpha^{-1}\cO^\times}\colon \alpha^{-1}\cO^\times\to \cO^\times$ is an isomorphism.
\end{definition}

The notion of a log structure on a (locally) ringed topos has already appeared a few times in the literature. For example, see Chapter 12 of \cite{gabber-ramero}. Our notion (and the basic theory that we will develop) agrees with the usual one for locally ringed topoi: if $(\fX,\cO)$ is a locally ringed topos, then it also gives a monoided topos, by regarding $\cO$ as a monoid with respect to multiplication. A Kato (pre-)log structure on this monoided topos is exactly the same as a (pre-)log structure on the locally ringed topos $(\fX,\cO)$, as in \cite[Definition 12.1.1]{gabber-ramero}. Our definition is more general; it applies for example to the context of positive topological spaces (see Section \ref{sec:2}, or \cite[Section 5.7]{molcho}), where the ``structure sheaf'' is not a sheaf of rings, but only of monoids.

A general definition of log structures in the sense of Kato in a ``category of spaces'' appears in \cite{molcho}, although the formalism considered in that article does not allow for Grothendieck topologies that are not ``classical'' (for example the \'etale topology on a scheme).

There is a category of Kato log structures on a monoided topos $(\fX,\cO)$. Maps $(M,\alpha)\to (N,\beta)$ are homomorphisms of monoids $\phi\colon M\to N$ such that the composite $\beta\circ\phi$ coincides with $\alpha$.

As in the algebraic case, one can obtain a Kato log structure from a Kato pre-log structure by replacing $M$ with the amalgamated sum $M\oplus_{\alpha^{-1}\cO^\times}\cO^\times$ and the induced map to $\cO$. A Kato log structure $(M,\alpha)$ on $(\fX,\cO)$ will be called \emph{quasi-integral} if the action of $\cO^\times$ on $M$ is free (i.e. does not have any non-trivial stabilizer). The quotient $\overline{M} = M/\cO^\times$ is usually called the \emph{characteristic monoid} of the log structure.

Let $\Div_{\fX}$ denote the quotient stack $[\cO/\cO^\times]$ over the site $(\fX,\cO)$, where the group $\cO^\times$ acts on $\cO$ via the inclusion. This is a symmetric monoidal stack, with monoidal structure induced by the multiplication of $\cO$. The notation is taken from \cite{borne-vistoli} (see in particular Examples 2.8 and Remark 2.9), and reflects the fact that in the algebraic case $[\cO/\cO^\times]$ is the stack of ``generalized Cartier divisors'' (i.e. line bundles with a global section).

\begin{definition}
A \emph{Deligne-Faltings} (DF for short) \emph{log structure} on a monoided topos $(\fX,\cO)$ is a pair $(A,L)$ where $A$ is a sharp monoid in $\fX$ and $L\colon A\to \Div_{\fX}$ is a symmetric monoidal functor with trivial kernel.
\end{definition}

In the definition above, ``with trivial kernel'' means that if a section $a\in A(U)$ is such that $L(a)$ is invertible in $\Div_{\fX}$ (in the monoidal sense), then $a=0$. This definition is a particular case of that of a ``Deligne-Faltings object'' of \cite{borne-vistoli}.

There is a category of DF log structures on a monoided topos $(\fX,\cO)$. Maps $(A,L)\to (B,K)$ are given by homomorphisms of monoids $\phi\colon A\to B$ with a natural equivalence $\Phi\colon K\circ \phi\cong L$.

\begin{proposition}\label{prop:katovsDF}
Let $(\fX,\cO)$ be a monoided topos. Then there is an equivalence of categories between quasi-integral Kato log structures and DF log structures on $(\fX,\cO)$.
\end{proposition}

\begin{proof}
Given a quasi-integral Kato log structure $(M,\alpha)$, we produce a DF log structure by dividing $\alpha\colon M\to \cO$ by the action of $\cO^\times$ in the stacky sense. We obtain a morphism of symmetric monoidal stacks $\overline{\alpha}\colon \overline{M}=[M/\cO^\times]\to [\cO/\cO^\times]=\Div_{\fX}$, and by assumption $\overline{M}$ is equivalent to a sheaf. The pair $(\overline{M},\overline{\alpha})$ is a DF log structure.

Conversely, starting from a DF log structure $(A,L)$ we consider the fibered product $M=A\times_{\Div_{\fX}}\cO$ and the induced symmetric monoidal morphism $\alpha\colon M\to \cO$. Since $\cO\to \Div_{\fX}$ is a $\cO^\times$-torsor, the stack $M$ is equivalent to a sheaf (of monoids), and the pair $(M,\alpha)$ is a quasi-integral Kato log structure.

We leave it to the reader to check that these two constructions are quasi-inverses (see also \cite[Theorem 3.6]{borne-vistoli}).
\end{proof}

For the rest of the paper all Kato log structures will be quasi-integral,  we will drop the ``Kato'' and ``DF'', and just talk about log structures, switching freely between the two notions and notations.

Next, we point out that log structures can be pulled back along morphisms of monoided topoi. Let $(\fX,\cO_{\fX})$ and $(\fY,\cO_{\fY})$ be monoided topoi. 

\begin{definition}
A morphism of monoided topoi $(\fX,\cO_{\fX})\to (\fY,\cO_{\fY})$ is a morphism of topoi $F = (F_{*}, F^{-1})\colon \fX\to \fY$ together with a homomorphism of sheaves of monoids $F^{-1}\cO_\fY\to \cO_\fX$.
\end{definition}

Morphisms of monoided topoi can be composed in the obvious way.

Given a morphism $f\colon (\fX,\cO_{\fX})\to (\fY,\cO_{\fY})$ and a log structure $(M,\alpha)$ on $(\fY,\cO_{\fY})$, there is a log structure $f^{-1}(M,\alpha)$ obtained by pullback: the morphism $\alpha\colon M\to \cO_\fY$ induces a pre-log structure $f^{-1}\alpha\colon f^{-1}M\to f^{-1}\cO_\fY\to \cO_\fX$, and $f^{-1}(M,\alpha)$ is the associated log structure.

\begin{remark}
If one uses DF log structures, the pullback is simply defined by the composite $f^{-1}A\to f^{-1}\Div_{\cY}\to \Div_{\cX}$ (see \cite[Section 3.2]{borne-vistoli} for details---the arguments generalize without any difficulty), and there is no need to change the sheaf of monoids.
\end{remark}

If $(\fX,\cO_{\fX})$ is equipped with a log structure $(M,\alpha)$ and $(\fY,\cO_{\fY})$ with $(N, \beta)$, a morphism between the ``log monoided topoi'' $(\fX,\cO_{\fX}, M, \alpha)\to (\fY,\cO_{\fY}, N, \beta)$ is a morphism of monoided topoi $f\colon (\fX,\cO_{\fX})\to (\fY,\cO_{\fY})$ together with a morphism of log structures $f^{-1}(N,\beta)\to (M,\alpha)$.

Such a morphism of log monoided topoi is said to be $\emph{strict}$ if $f^{-1}(N,\beta)\to (M,\alpha)$ is an isomorphism.

\section{Coherent sheaves of monoids and charts}\label{subsubsec:coherent}

Let $A$ be a monoid in a topos $\fX$, $P$ a finitely generated monoid and $P \arr A(\fX) = \Hom(\pt, A)$ a homomorphism of monoids. If $\underline{P}$ is the locally constant sheaf given by $P$, there is an induced homomorphism of sheaves of monoids $\phi\colon \underline P \arr A$.

\begin{definition}\hfil

\begin{enumerate1}

\item A homomorphism $P \arr A(\fX)$ is a \emph{chart} if the induced homomorphism $\phi\colon \underline P \arr A$ is a cokernel, or, equivalently, if the induced homomorphism $\underline P/\ker \phi \arr A$ is an isomorphism.

\item An \emph{atlas} for $A$ is a covering $\{U_{i}\}$ of $\fX$ together with a chart $P_{i} \arr A(U_{i})$ for the pullback of $A$ to $(\fX/U_{i})$ for each $i$.

\item The object $A$ is called \emph{coherent} if it admits an atlas.
\end{enumerate1}

\end{definition}

\begin{definition}
A monoid $A$ in $\fX$ is called \emph{sharp}, or \emph{integral}, or \emph{torsion-free}, or \emph{saturated}, if for each object $U$ of $\fX$ the monoid $A(U)$ has the homonymous property.
\end{definition}

As usual, ``fine'' will mean ``integral and finitely generated'' (and also coherent, when applied to a monoid $A$ in $\fX$).

It is easy to see that if $\fX$ has enough points, then $A$ has any of the properties above if and only if for each point $p$ of $\fX$, the stalk $A_{p}$ has the homonymous property.

\begin{proposition}\label{prop:extend-to-chart}
Let $A$ be a sharp coherent monoid in $\fX$. Assume that $A$ is fine, or fine and torsion-free, or fine and saturated. Then $A$ has an atlas $\{P_{i} \arr A(U_{i})\}$ in which every $P_{i}$ has the homonymous property.
\end{proposition}

If $\fX$ has enough points then one can give a proof along the lines of \cite[Proposition~3.15]{borne-vistoli}. 

\begin{proof}
Assume that $\phi\colon P\to A(U)$ is a chart for $A$ on $U$, and set $Q=A(U)$. We will show that $\psi\colon Q\to A(U)$ (the identity) is also a chart for $A$.

We have a diagram $\underline{P}/\ker \phi \to \underline{Q}/\ker\psi \to A$, and the composite is an isomorphism by assumption, so $\underline{Q}/\ker\psi \to A$ is surjective. Let us check that it is also injective.

For an object $V$ (on which we will localize further  several times), two sections $x$ and $y$ of $(\underline{Q}/\ker\psi)(V)$ that map to the same section of $A(V)$ are locally represented by two elements $a,b$ of $Q$, that map to the same element in $A(V)$. Since $\underline{P}\to A$ is surjective, locally the sections $a,b\in Q=A(U)$ are images of $p, p'\in P$. Since $\underline{P}/\ker\phi\to A$ is a cokernel and $p$ and $p'$ map to the same element in $A(V)$, there locally exist sections $h,k \in (\ker\phi)(V)$ such that $p+h=p'+k$ in $\underline{P}(V)$. If $h'$ and $k'$ are the images of $h$ and $k$ in $\underline{Q}(V)$ (and note that they are sections of $\ker\psi$), we have $a|_{V}+h'=b|_{V}+k'$. This implies that the two sections $x, y \in (\underline{Q}/\ker\psi)(V)$ considered at the beginning are locally equal, and shows that $\underline{Q}/\ker\psi \to A$ is injective.
\end{proof}

Now let $(\fX, \cO)$ be a monoided topos. 

\begin{definition}
A DF log structure $(A, L)$ on $(\fX, \cO)$ is \emph{coherent}, or \emph{fine}, or \emph{torsion-free}, or \emph{saturated}, if $A$ has the homonymous property.
\end{definition}

Notice that if $(A, L)$ is a DF log structure in $(\fX, \cO)$, then $A$ is automatically sharp. Hence, since every sharp fine saturated monoid is torsion-free, a fine saturated DF log structure is automatically torsion-free.

Given a symmetric monoidal functor $h\colon P\to \Div_\fX(\fX)$, by sheafifying and killing the kernel (in the monoidal sense) of the resulting functor we obtain a DF log structure $(P,h)^a$ (see Propositions 2.4 and  2.10 of \cite{borne-vistoli}). If $(A, L)$ is a DF log structure on $(\fX, \cO)$, $k\colon P \arr A(\fX)$ is a homomorphism of monoids, and $h = L \circ k\colon P \arr \Div_\fX(\fX)$ is the composite, we obtain a morphism of DF log structures $(P,h)^a \arr (A, L)$.

\begin{proposition}
The homomorphism $h\colon P \arr A(\fX)$ is a chart if and only if the induced morphism $(P,h)^a \arr (A, L)$ is an isomorphism.
\end{proposition}

\begin{proof}
Let $K\subseteq \underline P$ be the kernel of $h$, which equals the kernel of the symmetric monoidal functor $\underline P \arr \Div_{\fX}$ induced by $h$; by construction, the sheaf of monoids of the DF log structure is $\underline P/K$. Hence, $h$ is a cokernel if and only if the morphism of DF log structures $(P,h)^a \arr (A, L)$ induces an isomorphism on sheaves of monoids. On the other hand since $\Div_{\fX}$ is fibered in groupoids we have that a morphism of DF log structures $(B, M) \arr (A, L)$ is an isomorphism if and only if the corresponding homomorphism $B \arr A$ is an isomorphism, and this concludes the proof.
\end{proof}

In our general context Kato charts also work very well.

\begin{definition}
Let $(M, \alpha)$ be a Kato log structure on $(\fX, \cO)$.

\begin{enumerate1}

\item A \emph{Kato chart} on $(M, \alpha)$ is a homomorphism of monoids $P \arr M(\fX)$ such that composite $P \arr M(\fX) \arr \overline{M}(\fX)$ is a DF chart.

\item A \emph{Kato atlas} for $(M, \alpha)$ is a covering $\{U_{i}\}$ of $\fX$ together with a chart $P_{i} \arr M(U_{i})$ for the pullback of $(M, \alpha)$ to $(\fX/U_{i})$ for each $i$.

\end{enumerate1}
\end{definition}

One can give examples of fine log structures that do not admit Kato atlases; however, we have the following.

\begin{proposition}\label{prop:local-lifting}

Let $P$ be a fine torsion-free monoid. Any symmetric monoidal functor $P \arr \Div_{\fX}(\fX)$ lifts, locally on $\fX$, to a homomorphism $P \arr \cO(\fX)$.

\end{proposition}

\begin{proof}

Consider the induced homomorphism $\underline{P} \arr \Div_{\fX}$, and call the $E$ the pullback of $\cO$ to $\underline{P}$, which is an integral sheaf of monoids. Then $\cO^{\times}$ is contained in $E$, acts freely on $E$, and $E/\cO^{\times} = \underline{P}$. Passing to the associated sheaves of groups is a colimit-preserving functor, being a left adjoint. Furthermore, since $E$ is integral we have that the composite $\cO^{\times} \arr E \arr E^{\gp}$ is injective; so we have an exact sequence
   \[
   0 \arr \cO^{\times} \arr E^{\gp} \arr \underline{P}^{\gp} \arr 0\,.
   \]
But $\underline{P}^{\gp}$ is a free abelian group object in $\fX$, because $P$ is fine and torsion-free; hence by passing to a covering of $\fX$ we may assume that there is a splitting homomorphism $\underline{P}^{\gp} \arr E^{\gp}$. Since $E$ is an $\cO^{\times}$-torsor over $\underline{P}$, and $E^{\gp}$ is an $\cO^{\times}$-torsor over $\underline{P}^{\gp}$, we have that $E$ is the inverse image of $\underline{P} \subseteq \underline{P}^{\gp}$ inside $E^{\gp}$. It follows that the splitting $\underline{P}^{\gp} \arr E^{\gp}$ restricts to a splitting $\underline{P} \arr E$; the composite $\underline{P} \arr E \arr M$ gives a homomorphism $P \arr M(\fX)$, which is the desired Kato chart.
\end{proof}

\begin{corollary}
A fine torsion-free log structure has a Kato atlas.
\end{corollary}

\begin{proof}
Let $(M, \alpha)$ be a Kato log structure on $(\fX, \cO)$; call $(\overline{M}, L)$ the corresponding DF log structure. We may assume that $(\overline{M}, L)$ has a DF chart $P \arr \overline{M}(\fX)$; by Proposition~\ref{prop:extend-to-chart} we may assume that $P$ is integral and torsion-free. Then the result follows from Proposition~\ref{prop:local-lifting}.
\end{proof}

\subsection{Spaces of Kato charts and stacks of DF charts}\label{sec:charts.spaces}

\begin{assumption}\label{assumption:esp}
We fix a category $\esp$ admitting finite inverse limits, equipped with a subcanonical Grothendieck topology, a monoid object $\bA^1$ in $\esp$, and a class of arrows in $\esp$, called \emph{small}.

We assume the following conditions.

\begin{enumerate1}

\item An isomorphism is small.

\item The class of small arrows is closed under composition.

\item The class of small arrows is closed under pullback, that is, if $X \arr Y$ and $Y' \arr Y$ are arrows and $X \arr Y$ is small, the projection $Y'\times_{Y}X \arr Y'$ is also small.

\item If $U$ is an object of $\esp$, any covering of $U$ has a refinement $\{U_{i} \arr U\}$ in which every map $U_{i} \arr U$ is small.

\end{enumerate1}
\end{assumption}

We think about the objects of $\esp$ as some sort of ``spaces'' (as in \cite{molcho, molcho2}---the notation is borrowed from there). 

In Section \ref{sec:2}, we will consider in particular the following three classes of examples.

\begin{enumerate1}

\item $\esp$ is the category of schemes over a fixed base ring $k$. The topology is the étale topology, $\bA^{1} = \bA^{1}_{k}$, and small arrows are étale maps.

\item $\esp$ is the category of analytic spaces, the topology is the classical topology, $\bA^{1} = \bC$, and small arrows are open embeddings.

\item $\esp$ is the category of topological spaces, the topology is the classical topology, $\bA^{1} = \bC$, and small arrows are open embeddings.

\end{enumerate1}

Denote by $\cO$ the sheaf of morphisms towards $\bA^1$. If $\cO^{\times} \subseteq \cO$ is the subsheaf of invertible sections, then $\cO^{\times}$ is represented by a subobject $\gm \subseteq \bA^{1}$: if $\pt$ is the terminal object of $\esp$ and $\pt \arr \bA^{1}$ is the morphism corresponding to the identity in $\bA^{1}$, then $\gm$ is represented by the fibered product $(\bA^{1} \times \bA^{1})\times_{\bA^{1}}\pt$, where the morphism $\bA^{1} \times \bA^{1} \arr \bA^{1}$ is the operation in $\bA^{1}$.

Given an object $S\in \esp$, we can consider the induced small site $S\sm$ of $S$: its objects are small arrows $U \arr S$, with the obvious morphisms, and the topology is induced by the topology of $\esp$. We denote by $\fX_S$ the topos of sheaves on $S\sm$, together with the sheaf $\cO_S$ of morphisms to the monoid object $\bA^1$. This gives a monoided topos $(\fX_S,\cO_S)$, and we can consider log structures on it, morphisms between these objects, and so on.

If $f\colon S \arr T$ is an arrow in $\esp$, fibered product induces a functor $T\sm \arr S\sm$, which in turn gives a morphism of topoi $(f_{*}, f^{-1})\colon \fX_{S}\arr \fX_{T}$.

A \emph{log space} $(S,M,\alpha)$ in $\esp$ will be an object $S \in \esp$, together with a log structure on the monoided topos $\fX_S$ as above (we used the notation for a Kato log structure, but we can equivalently use DF log structures). A morphism $(S,M,\alpha)\to (T,N,\beta)$ of log spaces is a morphism $f\colon S\to T$ in $\esp$, together with a morphism $f^{-1}(N,\beta)\to (M,\alpha)$ of log structures over $S$ (i.e. for the topos $\fX_S$).

\subsubsection*{Spaces of Kato charts}

Suppose that $P$ is a finitely generated monoid. Consider the sheaf of monoids on $\esp$ that sends an object $S$ into the monoid $\Hom\bigl(P, \cO(S)\bigr)$; this is representable by a monoid object $\bA(P)$ in $\esp$, which is constructed as follows (see also \cite[Section 4.4]{molcho}, from where, again, we also borrow the notation). 

If $P = \bN^{r}$ is a free monoid, then $\Hom\bigl(P, \cO(S)\bigr) = \cO(S)^{r}$, so $\bA(P)= (\bA^{1})^{r}$. In general, by  R\'edei's Theorem \cite[Theorem 72]{redei} the monoid $P$ is finitely presented, so it is the coequalizer of two homomorphisms $\bN^{s} \double \bN^{r}$. Hence $\Hom\bigl(P, \cO(S)\bigr)$ is represented by the equalizer the two arrows of $\bA(\bN^{r}) \double \bA(\bN^{s})$ induced by the homomorphisms above, which exists by hypothesis.

The definition of $\bA(P)$ equips it with a natural log structure, induced by the universal homomorphism $P\to \cO(\bA(P))$. A morphism of log spaces $(S,M,\alpha) \to \bA(P)$ corresponds to compatible homomorphisms $P\to \cO(S)$ and $P\to M(S)$, and such a morphism is strict if and only if $P\to M(S)$ is a Kato chart for $(M,\alpha)$ on $S$.

The group of invertible elements in $\Hom\bigl(P, \cO(S)\bigr)$ is
   \begin{align*}
   \Hom\bigl(P, \cO^{\times}(S)\bigr) &= \Hom\bigl(P^{\gp}, \cO^{\times}(S)\bigr)\\
      &= \Hom\bigl(P^{\gp}, \cO(S)\bigr)\,
   \end{align*}
and hence the group functor that sends $S$ into $\Hom\bigl(P, \cO^{\times}(S)\bigr)$ is represented by $\bG(P) := \bA(P^{\gp})$. There is an obvious action of $\bG(P)$ on $\bA(P)$.

\subsubsection*{Stacks of DF charts}
Consider the stack $\Div = [\cO/\cO^{\times}] = [\bA^{1}/\gm]$ over the site $\esp$, and a finitely generated monoid $P$. There is a stack $\cA(P)$ over $\esp$, whose sections over an object $S$ of $\esp$ consist of symmetric monoidal functors $P \arr \Div(S)$. 

\begin{proposition}\label{prop:stack.DF.charts}
Suppose that $P$ is fine and torsion-free. Then the stack $\cA(P)$ is the quotient $[\bA(P)/\bG(P)]$.
\end{proposition}

This is the analogue of \cite[Proposition 3.25]{borne-vistoli}.

\begin{proof}
Denote by $\pi\colon \cO \arr \Div$ the projection. There is an obvious map $\pi_{P}\colon \bA(P) \arr \cA(P)$ that sends a homomorphism $\phi\colon P \arr \cO(S)$ into the composite $\pi\circ \phi$. Since $\pi_{P}\colon \bA(P) \arr \cA(P)$ is an epimorphism of stacks (that is, it admits local sections) by Lemma~\ref{prop:local-lifting} (this is where we are using the torsion-free condition), we have that $\cA(P)$ is the quotient of the groupoid
   \[
   \bA(P)\times_{\cA(P)} \bA(P)\double \bA(P)\,;
   \]
hence it enough to prove that the groupoid above is isomorphic to the groupoid
   \[
   \bA(P) \times \bG(P) \double \bA(P)
   \]
defined by the action of $\bG(P)$. 

Suppose that $(\phi, \alpha)$ is an element of $\bigl(\bA(P) \times \bG(P)\bigr)(S)$; in other words, $\phi\colon P \arr \cO(S)$ and $\alpha\colon P \arr \cO^{\times}(S)$ are homomorphisms of monoids. Then the product $\phi\alpha\colon P \arr \cO(S)$ is another section of $\bA(P)(S)$, and $\alpha$ gives an isomorphism of $\pi_{P} \circ \phi$ with $\pi_{P} \circ (\phi\alpha)$; so we obtain an element $(\phi, \phi\alpha, \alpha) \in \bigl(\bA(P)\times_{\cA(P)} \bA(P)\bigr)(S)$. This gives a map from $\bA(P) \times \bG(P)$ to  $\bA(P)\times_{\cA(P)} \bA(P)$, which is easily checked to be an isomorphism.
\end{proof}

A DF structure on an object $\cS$ of $\St_\esp$ (the 2-category of stacks over the site $\esp$) can be defined as a collection of DF structures on all pairs  $(S,f)$ of spaces $S\in \esp$ with a map $f\colon S\to \cS$, and compatible isomorphisms between the various pullbacks. The fact that $$\Hom_{\St_\esp}(S,\cA(P))=\Hom_\sym(P,\Div(S))$$ for $P$ torsion-free, equips $\cA(P)$ with a tautological DF structure, and makes it the stack of DF charts for log spaces in $\esp$. In other words, DF charts on a log space $(S,M,\alpha)$ in $\esp$ from a fine torsion-free monoid $P$ correspond exactly to strict morphisms $S\to \cA(P)$.

\begin{remark}
As suggested by the referee, it is natural to ask if there is a version of Olsson's classifying stack of log structures \cite{Ols} in this abstract context. In fact, the \'etale cover of $\cL og_S$ described in \cite[Section 5]{Ols} is made up precisely of the quotient stacks $ \cA(P)=[\bA(P)/\bG(P)]=\big[ \Spec \bZ[P]\,/\, \Spec \bZ[P^\gp]\big]$, in the context of log schemes over some base.

In our axiomatic context, for a fine saturated log space $(S,M,\alpha)$ in $\esp$ one can certainly consider the stack $\cL og_{(S,M,\alpha)}$ over $\esp$ classifying fine saturated log spaces over $(S,M,\alpha)$. Its objects are morphisms of fine saturated log spaces $(T,N,\beta)\to (S,M,\alpha)$, and arrows are strict morphisms of log spaces, compatible with the maps to $(S,M,\alpha)$. Our reason for restricting to saturated log structures right away (contrarily to what is done in \cite{Ols}) is the presence of the torsion-freeness assumption in Proposition \ref{prop:stack.DF.charts}. It is possible that this assumption can be removed.

For every small map $U\to S$, fine monoid $P$ with a chart $P\to M(U)$ for the log space $(U,M|_U, \alpha|_U)$, and morphism of fine saturated monoids $P\to Q$, there is a natural morphism of stacks $\cA(Q)\times_{\cA(P)} U\to \cL og_{(S,M, \alpha)}$, sending an object $(T\to \cA(Q), T\to U)$ to the morphism of log spaces $(T, N, \beta)\to (U,M|_U, \alpha|_U)\to (S,M,\alpha)$, where $(N,\beta)$ is the log structure on $T$ induced by the map $T\to \cA(Q)$. Essentially by the existence of local charts and Proposition \ref{prop:stack.DF.charts}, the induced map
$$
\coprod_{(U,\,P\to M(U), \,P\to Q)} \cA(Q)\times_{\cA(P)} U \xrightarrow{\;\;\;\;\;\;\; \Phi \;\;\;\;\;\;\;}  \cL og_{(S,M, \alpha)}
$$
is surjective. In the case of log schemes, it is also representable and \'etale \cite[Corollary 5.25]{Ols}.

It is possible that something more can be said in general, but it seems likely that at this point one has to make more assumptions on the topology of $\esp$. For example, if $\esp$ is complex analytic spaces with the classical topology and small maps are open embeddings, the map $\Phi$ will most likely only be \'etale.

We stop here on this path for the moment, and leave it to possible future development.
\end{remark}

\section{Root stacks}\label{sec:root.stacks}

We can extend the definition of root stacks given in \cite{borne-vistoli} to the axiomatic setting. Let $(S, A, L)$ be a log space in a category $\esp$ as in Assumption \ref{assumption:esp}. Let $A \arr B$ be a homomorphism of sheaves of monoids in $\fX_{S}$.

The following definition is an obvious extension of \cite[Definition 4.16]{borne-vistoli} to our situation; see this paper if more details are needed.

\begin{definition}
The \emph{root stack} $\radice[B]{(S, A, L)}$ is the stack over $\esp$ defined as follows.

An object of $\radice[B]{(S, A, L)}$ over a space $T$ is a triple $(f, N, \alpha)$, where $f\colon  T \arr S$ is an arrow in $\esp$, $N\colon f^{-1}B\arr \Div_{T}$ is a symmetric monoidal functor with trivial kernel, and $\alpha$ is an isomorphism of log structures between the restriction of $(f^{-1}A, N|_{f^{-1}A})$ and $f^{-1}(A, L)$.

The arrows are defined in the obvious way.
\end{definition}

Notice that if $A \arr B \arr C$ are maps of sheaves of monoids on $S$, we have obvious functors $\radice[C]{(S, A, L)} \arr \radice[B]{(S, A, L)}$.

One is usually interested in the following two classes of examples.

Assume that $A$ is saturated. If $n$ is a positive integer, consider the subsheaf $\frac{1}{n}A \subseteq A^{\gp}\otimes \bQ$, with the natural embedding $A \subseteq \frac{1}{n}A$. (Of course we can think of $\frac{1}{n}A$ as $A$ itself, with the map $A \arr \frac{1}{n}A$ given by multiplication by $n$.) We denote the resulting stack by $\radice[n]{(S, A, L)} = \radice[n]S$. Notice that $\radice[n]S$ comes with a tautological log structure, whose sheaf of monoids is the pullback of $\frac{1}{n}A$ from $S$.

Here is a local description of $\radice[n]S$. Suppose that there is a chart $P \arr A(S)$, where $P$ is a fine saturated monoid. The embedding $P \subseteq \frac{1}{n}P$ gives compatible maps $\bA(\frac{1}{n}P) \arr \bA(P)$ and $\bG(\frac{1}{n}P) \arr \bG(P)$, which induce a morphism of stacks $\cA(\frac{1}{n}P) \arr \cA(P)$. The log structure on $\radice[n]S$, with sheaf of monoids equal to the pullback of $\frac{1}{n}A$, gives a map $\radice[n]S \arr \cA(\frac{1}{n}P)$ with an isomorphism between the two composites $\radice[n]S \arr \cA(\frac{1}{n}P) \arr \cA(P)$ and $\radice[n]S \arr S \arr \cA(P)$.

\begin{proposition}\label{prop:kummer.chart}
The diagram
   \[
   \xymatrix{
   {}\radice[n]S \ar[r]\ar[d] & {}\cA(\frac{1}{n}P) \ar[d]\\
   S \ar[r] & {}\cA(P)
   }
   \]
is cartesian.
\end{proposition}

The proof is straightforward (cfr. \cite[Proposition 4.13]{borne-vistoli}), and the analogous result holds for Kummer extensions $P\to Q$ not necessarily of the form $P\subseteq \frac{1}{n}P$.

In particular, if the chart $S\to \cA(P)$ factors through a Kato chart $S\to \bA(P)$, we have an isomorphism $\radice[n]{S}=[S\times_{\bA(P)} \bA(\frac{1}{n}P)/\mu_n(P)]$, where $\mu_n(P)$ is the group object $\bG(C_n(P))$ associated with the cokernel $C_n(P)$ of the homomorphism $P^\gp\to \frac{1}{n}P^\gp$, and the action on $S\times_{\bA(P)} \bA(\frac{1}{n}P)$ is induced by the action on the second factor.

Given a sheaf of integral monoids $A$ on $S$, consider the submonoid $A_{\bQ} \arr A^{\gp}\otimes\bQ$, defined as the subsheaf of sections $a$ of $A^{\gp}\otimes\bQ$ with the property that locally on $S$ there is a positive natural number $q$ such that $qs$ is in $A \subseteq A^{\gp}\otimes\bQ$. The following is the analogue of \cite[Definition~3.3]{TV}.

\begin{definition}
The \emph{infinite root stack} $\radice[\infty]{(S, A, L)} = \radice[\infty]S$ is $\radice[A_{\bQ}]{(S, A, L)}$. 
\end{definition}

We have the following alternate characterization. The embeddings $\frac{1}{n}A \subseteq A_{\bQ}$ induce compatible functors $\radice[\infty]S \arr \radice[n]S$, and therefore a morphism $\radice[\infty]S \arr \projlim_{n}\radice[n]S$ (here the indexing system for the projective limit is the set of positive integers, ordered by divisibility). The following is a straightforward generalization of \cite[Proposition~3.5]{TV}.

\begin{proposition}
The morphism $\radice[\infty]S \arr \projlim_{n}\radice[n]S$ above is an equivalence of stacks over $\esp$.
\end{proposition}

It should be remarked that $\radice[\infty]S$ is not always well behaved; when the category $\esp$ is too small, for example it does not admit countable projective limits, then $\radice[\infty]S$ may have too few objects (see \cite[Remark 2.14]{TVnew}
). In these cases one should define the infinite root stack to be the projective system given by the $\radice[n]S$.

\section{Examples}\label{sec:2}
To conclude, we briefly discuss some concrete situations where the formalism developed in this paper applies, and one application, described in detail in \cite{TVnew}.

From now on log structures will be fine and saturated. In particular we can find local charts from monoids that are fine, saturated and sharp, hence torsion-free. 
In the analytic and topological case, root stacks will be analytic and topological stacks respectively. These are the analogue, in the corresponding setting, of algebraic stacks. For a brief discussion, see for example \cite[Appendix A]{knvsroot} or \cite[Section 2.3]{TVnew}.

\subsection*{Algebraic}

In the algebraic case we are taking $\esp$ to be the the category of schemes (or more generally algebraic spaces) over a ring $k$ (or even a more general base) with the \'etale topology, and $\bA^1_k$ as monoid object (the operation is given by multiplication). The scheme $\bA(P)$ is $\Spec k[P]$, and the stack $\cA(P)$ is the quotient $[\Spec k[P]/\widehat{P}]$, where $\widehat{P}$ is the diagonalizable group associated with $P^\gp$ (these are exactly the stacks of charts of \cite{borne-vistoli}).

If $X$ is a fine saturated log algebraic space, the root stack $\radice[n]{X}$ is an algebraic stack (Deligne--Mumford in good characteristic), and has local presentations where $X$ has a Kato chart $X\to \Spec k[P]$, as the quotient $[X \times_{\Spec k[P]}\Spec k[\frac{1}{n}P]/\mu_n(P)]$, where $\mu_n(P)$ is the Cartier dual of the cokernel of $P^\gp\to \frac{1}{n}P^\gp$ and the action on $X \times_{\Spec k[P]}\Spec k[\frac{1}{n}P]$ is induced by the one on the second factor (cfr. \cite[Proposition 4.13]{borne-vistoli}).

In this case the infinite root stack $\radice[\infty]X$ is a pro-algebraic stack, and has been studied extensively in \cite{TV}.

\subsection*{Analytic}

In this case, in the notation of (\ref{sec:charts.spaces}) we are considering $\esp$ to be the category of complex analytic spaces endowed with the classical topology, and as monoid object the analytic space $\bC$ (with multiplication). The realization $\bA(P)$ of a monoid $P$ is the analytic space $(\Spec \bC[P])_\an$, and the stack $\cA(P)$ is the quotient $[(\Spec \bC[P])_\an/\widehat{P}_\an]$.

In analogy to the algebraic case, it is not hard to prove that the stack $\Div_X^\an=[\cO_X/\cO_X^\times]=[\bC/\bC^\times]_X$ parametrizes pairs $(L,s)$ of a holomorphic line bundle with a global section, over the classical site $\cA_X$.

In this case if $X$ is a fine and saturated log analytic space, for every $n$ the root stack $\radice[n]{X}$ is an analytic Deligne--Mumford stack, and locally where $X$ has a Kato chart $X\to (\Spec \bC[P])_\an$ is isomorphic to the quotient $[X \times_{(\Spec\bC[P])_\an}(\Spec\bC[\frac{1}{n}P])_\an/(\mu_n(P))_\an]$, where $\mu_n(P)$ is the Cartier dual of the cokernel of $P^\gp\to \frac{1}{n}P^\gp$ and the action on $X \times_{(\Spec\bC[P])_\an}(\Spec\bC[\frac{1}{n}P])_\an$ is induced by the one on the second factor.

\begin{remark}\label{rmk:compatible.analytification}
Let $X$ be a scheme locally of finite type over $\bC$, and consider the analytification $X_\an$, a complex analytic space. Then there is a morphism of monoided topoi $(X_\an,\cO_{X_\an})\to (X_\et,\cO_{X_\et})$, along which we can pullback log structures, as described at the end of Section \ref{sec:topoi}. It is clear that if $X\to \Spec \bC[P]$ is a Kato chart for $X$, the induced $X_\an\to (\Spec \bC[P])_\an$ will be a Kato chart for $X_\an$.

Moreover, note that by the above description and since analytification respects colimits, in presence of a Kato chart the formation of root stacks is compatible with analytification, in the sense that $(\radice[n]{X})_\an\cong \radice[n]{X_\an}$.

In fact, this is true for any fine saturated $X$, see \cite[Proposition 2.15]{TVnew}.
\end{remark}

\subsection*{Topological}

For the topological case we take as $\esp$ the category of topological spaces with the classical topology, and as monoid object the topological space $\bC$, with multiplication. One can consider variants of this by picking a different monoid object, for example $\bR$ with multiplication. The one we consider here is more natural for some purposes, because a log structure on a complex analytic space (in the above sense) will induce a ``topological'' log structure on the underlying topological space.

The realization $\bA(P)$ of a monoid $P$ is the topological space $(\Spec \bC[P])_\top$ (i.e. the space of $\bC$-valued points of $\Spec \bC[P]$, equipped with the analytic topology), and the stack $\cA(P)$ is the quotient $[(\Spec \bC[P])_\top/\widehat{P}_\top]$. 
In this setting the stack $[\cO/\cO^\times]$ is the quotient stack $\Div_T^\top=[\bC/\bC^\times]_T$ on the classical site of the topological space $T$ (here $\bC$ and $\bC^\times$ are seen just as topological objects), which turns out to parametrize continuous complex line bundles with a global section, in analogy with the algebraic and analytic cases.

If $T$ is a fine saturated log topological space, for every $n$ the $n$-th root stack $\radice[n]{T}$ is a topological Deligne--Mumford stack, and where $T$ has a Kato chart $T\to (\Spec \bC[P])_\top$ is isomorphic to the quotient stack $[T \times_{(\Spec \bC[P])_\top}(\Spec \bC[\frac{1}{n}P])_\top/(\mu_n(P))_\top]$, where $\mu_n(P)$ is the Cartier dual of the cokernel of $P^\gp\to \frac{1}{n}P^\gp$ and the action on $T \times_{(\Spec \bC[P])_\top}(\Spec \bC[\frac{1}{n}P])_\top$ is given by the one on the second factor.

The same observations of Remark \ref{rmk:compatible.analytification} apply to this case as well.

\subsection*{Positive topological spaces}

The formalism applies also to the cases considered in \cite{molcho}. For example, assume that $\esp$ is the category of topological spaces, and the monoid object is $\bR_{\geq 0}$ with multiplication (cfr. \cite[Section 5.7]{molcho}). In this case, the morphism $\alpha$ of a log structure takes value in the sheaf of real-valued positive functions of a space $T$, and the stack $\Div_S$ on an object $S\in \esp$ is the quotient $[\bR_{\geq 0}/\bR_{>0}]|_S$. This parametrizes real oriented line bundles.

The realization $\bA(P)$ of the monoid $P$ in $\esp$ is the space $\bR_{\geq 0}(P)=\Hom(P,\bR_{\geq 0})$, and the stack $\cA(P)$ is the quotient $[\bR_{\geq 0}(P)/\bR_{> 0}(P)]$. Note in particular that, by Proposition \ref{prop:stack.DF.charts}, for $P$ fine and torsion-free there is an equivalence between maps $S\to [\bR_{\geq 0}(P)/\bR_{> 0}(P)]$ and symmetric monoidal functors $P\to [\bR_{\geq 0}/\bR_{>0}]|_S$.

In this case root stacks are uninteresting, because raising to the $n$-th power $\bR_{\geq 0}\to \bR_{\geq 0}$ is a homeomorphism.

\subsection*{Comparison between the Kato--Nakayama space and the infinite root stack}

{To conclude, we give a brief overview of the contents of \cite{TVnew}. The present article was originally included in that paper, as a section introducing the necessary formalism.
}

If $X$ is a log scheme locally of finite type over $\bC$, there is a construction due to Kato and Nakayama \cite{KN} and now called the \emph{Kato--Nakayama space}, of a topological space $X_\log$ that can be though of as the ``analytification'' of the log scheme $X$. In good cases, $X_\log$ is a manifold with boundary, and the boundary lives on the locus of $X$ where the log structure is non-trivial.

The analogy with the infinite root stack can be seen by looking at the fibers of the projections $X_\log\to X$ and $\radice[\infty]{X}\to X$: the fibers of the first map are real tori $(S^1)^r$, where $r$ is the rank of the stalk of the group $\overline{M}^\gp$, and the reduced fibers of the second are classifying stacks $\mathrm B\widehat{\bZ}^r$ for the same $r$. By seeing $S^1$ as the classifying space $\mathrm B\bZ$, we can think of $\mathrm B\widehat{\bZ}^r$ as the profinite completion (in the appropriate sense) of $(S^1)^r$. In fact, as proven in \cite[Theorem 6.4]{knvsroot}, for every fine saturated log scheme locally of finite type over $\bC$, there is a canonical morphism of {topological stacks} $X_\log\to \sqrt[\infty]{X}_\top=\varprojlim_n \radice[n]{X}_\top$, that induces an equivalence between profinite completions.

One way to construct this morphism \cite[Section 3.4]{TVnew} is to find a functorial description of the topological stack $\radice[n]{X}_\top$ obtained from the complex-analytic root stack $\radice[n]{X}$, and show that the Kato--Nakayama space naturally carries an object of the kind parametrized by this stack. As shown in \cite[Section 2.5]{TVnew}, it turns out, quit pleasingly, that $\radice[n]{X}_\top$ can be naturally identified with the \emph{topological} root stack of the topological space $X_\top$, equipped with the induced topological log structure. The need to have a basic theory of this sort of log structures and root stacks on topological and analytic spaces was what drove us to writing things in the generality of the present paper in the first place.


\section*{Conflict of interest statement}

On behalf of all authors, the corresponding author states that there is no conflict of interest.


\bibliographystyle{alpha}
\bibliography{biblio}

\end{document}